\documentclass[11pt]{amsart}
\usepackage{amsmath}
\usepackage{amssymb}
\usepackage{latexsym}
\usepackage{graphicx}

\newcommand{\N}{\mathbb{N}}

\newtheorem{theorem}{Theorem}%[section]
\newtheorem{corollary}{Corollary}%[section]
\newtheorem{lemma}{Lemma}%[section]
\theoremstyle{remark}
\newtheorem{remark}{Remark}%[section]
\begin{document}

\title[Subclasses of Analytic and Univalent Functions]{\large New Subclasses of Analytic and Univalent Functions Involving Certain Convolution Operators}

\author[K. O. Babalola]{K. O. BABALOLA}

\begin{abstract}
Let $E$ be the open unit disk $\{z\in \mathbb{C}: |z|<1\}$. Let
$A$ be the class of analytic functions in $E$, which have the form
$f(z)=z+a_2z^2+...$. We define operators $L_n^\sigma\colon A\to A$
using the convolution $*$. Using these operators, we define and
study new classes of functions in the unit disk. Moreover, we
obtain some basic properties of the new classes, namely inclusion,
growth, covering, distortion, closure under certain integral
transformation and coefficient inequalities.
\end{abstract}

%\begin{amssubject}
%30C45.
%\end{amssubject}

%\begin{keyword}
%Convolution operators, analytic and univalent functions.
%\end{keyword}

\maketitle

\section{Introduction}
Denote by $A$ the class of functions
\[
f(z)=z+a_2z^2+...
\]
which are analytic in $E$. Let $P$ be the class of functions
\begin{equation}
 p(z)=1+c_1z+c_2z^2+...\, \label{1}
\end{equation}
which are also analytic in the unit disk $E$ and satisfy
$Re\;p(z)>0$, $z\in E$. Furthermore, for $0\le\beta<1$, let
$P(\beta)$ denote the subclasses of $P$ consisting of analytic
functions of the form $p_{\beta}(z)=\beta+(1-\beta)p(z)$, $p\in
P$.

It is well known  that a function $f\in A$ is said to belong to
the class $S_0(\beta)$ if $f(z)/z\in P(\beta)$, and is said to be
of bounded turning of order $\beta$ if $f'(z)\in P(\beta)$. The
class of functions of bounded turning of order $\beta$ is denoted
by $R(\beta)$ and is known to consist only of univalent functions
in the unit disk. These classes of functions were studied in the
literatures \cite{THM, KY} and various generalizations of them
have appeared in \cite{SA, BO, GS, TOO}.

Let $g(z)=z+b_2z^2+... \in A$. The convolution (or Hadamard
product) of $f$ and $g$ (written as $f*g$) is defined as
\[
(f*g)(z)=z+\sum_{k=2}^\infty a_kb_kz^k.
\]
Let $\sigma$ be a fixed real number and $n\in \N$. Define
\[
\tau_{\sigma,n}(z)=\frac{z}{(1-z)^{\sigma
-(n-1)}},\;\;\;\sigma-(n-1)>0
\]
and $\tau_{\sigma,n}^{(-1)}$ such that
\[
(\tau_{\sigma,n}*\tau_{\sigma,n}^{(-1)})(z)=\frac{z}{1-z}.
\]
For $n=0$, we simply write $\tau_\sigma$ and $\tau_\sigma^{(-1)}$
respectively. Let $f\in A$, define the operator $D^\sigma\colon
A\to A$ by
\[
D^\sigma f(z)=(\tau_\sigma*f)(z).
\]
The operator $D^\sigma$ is called the Ruscheweyh derivative
\cite{SR}. Analoguos to $D^\sigma$, Noor \cite{KIN} defined the
integral operator $I_\sigma\colon A\to A$ by
\[
I_\sigma f(z)=(\tau_\sigma^{(-1)}*f)(z).
\]
The operators $D^\sigma$ and $I_\sigma$ have been used to define
several classes of functions (see \cite{SA, BO, GS, TOO, SR,
GSS}). We define the following operators.

\medskip

 {\sc Definition 1.} Let $f\in A$. We define the operators
$L_n^\sigma:A\rightarrow A$ as follows:
\[
L_n^\sigma f(z)=(\tau_\sigma*\tau_{\sigma,n}^{(-1)}*f)(z).
\]

\medskip

 {\sc Definition 2.} Let $f\in A$. We define the operators
$l_n^\sigma:A\rightarrow A$ as follows:
\[
l_n^\sigma f(z)=(\tau_{\sigma}^{(-1)}*\tau_{\sigma,n}*f)(z).
\]

Note that $L_0^\sigma f(z)=L_0^0 f(z)=f(z)$, $L_1^1
f(z)=zf^{\prime}(z)$. Furthermore $L_n^n f(z)=D^n f(z)$ and
$L_{-n}^0 f(z)=I_n f(z)$. Similarly, $l_0^\sigma f(z)=l_0^0
f(z)=f(z)$, $l_1^1 f(z)=zf^{\prime}(z)$, $l_n^n f(z)=I_n f(z)$ and
$l_{-n}^0 f(z)=D^n
f(z)$. We also have the following remark.\\
\begin{remark}
Let $f\in A$. Then
\[
L_n^\sigma (l_n^\sigma f(z))=l_n^\sigma (L_n^\sigma f(z))=f(z).
\]
\end{remark}
In the case $\sigma=n$ we write $L_nf(z)(=D^nf(z))$ instead of
$L_n^nf(z)$ and $l_nf(z)(=I_nf(z))$ instead of $l_n^nf(z)$.

Next we isolate new classes of functions by:
\medskip

 {\sc Definition 3.} Let $f\in A$. Let $\sigma$ be any fixed real
number satisfying $\sigma-(n-1)>0$ for $n\in \mathbb{N}$. Then for
$0\le\beta<1$, a function $f\in A$ is said to be in the class
$B_n^\sigma (\beta)$ if and only if
\begin{equation}
Re\frac{L_n^\sigma f(z)}{z}>\beta,\;\;\;z\in E\,. \label{2}
\end{equation}
If $\sigma=n$ we write $B_n(\beta)$ in place of $B_n^\sigma
(\beta)$. We also note the following equivalent classes of
functions: $B_0(\beta)\equiv S_0(\beta)$ and $B_1(\beta)\equiv
R(\beta)$. In \cite{GS}, Goel and Sohi defined classes
$M_n(\beta)$ as consisting of functions $f\in A$ satisfying
\[
Re\frac{D^{n+1}f(z)}{z}>\beta,\;\;\;z\in E.
\]
These classes coincide with $B_{n+1}^\sigma (\beta)$. Following
from the geometric condition ~(\ref{2}) and Remark 1.3, functions
in the classes $B_n^\sigma (\beta)$ can be represented in terms of
functions in $P(\beta)$ as
\[
f(z)=l_n^\sigma [zp_\beta (z)].
\]
We investigate the classes $B_n^\sigma (\beta)$ in Section 3.
However, we require some preliminary discussions and results,
which we present in the next section.

\section{Two-parameter integral iteration of the class {\em P}}
In \cite{BO}, the authors identified the following iterated
integral transformation of functions in the class $P$.

\medskip

 {\sc Definition 4.}(\cite{BO}) Let $p\in P$ and $\alpha>0$ be real. The
{\em nth} iterated integral transform of $p(z), z\in E$ is defined
as
\[
p_n(z)=\frac{\alpha}{z^\alpha}\int_0^zt^{\alpha-1}p_{n-1}(t)dt,\;\;\;
n\geq 1
\]
with $p_0(z)=p(z)$.

The transformation, denoted by $P_n$, arose from the study of
classes $T_n^\alpha (\beta)$ consisting of functions defined by
the geometric condition
$Re\{D^nf(z)^\alpha/\alpha^nz^\alpha\}>\beta$, where $\alpha>0$ is
real, $0\le\beta<1$, and $D^n\;(n\in \N)$ is the Salagean
derivative operator defined as $D^0f(z)=f(z)$ and
$D^nf(z)=z[D^{n-1)}f(z)]'$ (see \cite{SA, BO, TOO}); and was
applied successfully in providing elegant proofs of many results.
It is known that for each $n\geq 1$, the class $T_n^\alpha
(\beta)$ consists only of univalent functions in the unit disk. A
basic relationship between the classes $P_n$ and $T_n^\alpha
(\beta)$ was given by the following lemma.
\begin{lemma}(\cite{BO}) Let $f\in A$, and $\alpha$, $\beta$ and
$D^n$ as defined above. Then the following are equivalent:

{\rm(i)} $f\in T_n^\alpha (\beta)$,

{\rm(ii)}$(D^nf(z)^\alpha/\alpha^nz^\alpha-\beta)/(1-\beta)\in P$,

{\rm(iii)} $(f(z)^\alpha/z^\alpha-\beta)/(1-\beta)\in P_n$.
\end{lemma}
Analogous to $P_n$ we define the following two-parameter
integral iteration of a $p\in P$.\\
\medskip

 {\sc Definition 5.} Let $p\in P$. Let $\sigma$ be any fixed real
number such that $\sigma-(n-1)>0$ for $n\in \mathbb{N}$. We define
the {\em sigma-nth} integral iteration of $p(z)$, $z\in E$ as
\begin{equation}
p_{\sigma,n}(z)=\frac{\sigma-(n-1)}{z^{\sigma-(n-1)}}\int_0^zt^{\sigma-n}p_{\sigma,n-1}(t)dt,\;\;\;
n\geq 1\, \label{3}
\end{equation}
with $p_{\sigma,0}(z)=p(z)$.

We note that since $p_{\sigma,0}(z)$ belongs to $P$, the transform
$p_{\sigma,n}(z)$ is analytic, and $p_{\sigma,n}(0)=1$ and
$p_{\sigma,n}(z)\neq 0$ for $z\in E$. We denote the family of
iterations above by $P_n^\sigma$. With $p(z)$ given by ~(\ref{1})
it is easily verified that
\[
p_{\sigma,n}(z)=1+\sum_{k=1}^\infty c_{n,k}^\sigma z^k
\]
where
\begin{equation}
c_{n,k}^\sigma=\frac{\sigma
(\sigma-1)...(\sigma-(n-1))}{(\sigma+k)(\sigma+k-1)...(\sigma+k-(n-1))}c_k,\;\;\;
k\geq 1\,. \label{4}
\end{equation}
Observe that the multiplier of $c_k$ in ~(\ref{4}) can be written
in factorial form as:
\[
\frac{\sigma
(\sigma-1)...(\sigma-(n-1))}{(\sigma+k)(\sigma+k-1)...(\sigma+k-(n-1))}=\frac{\sigma
!}{(\sigma+k)!}\frac{(\sigma+k-n)!}{(\sigma-n)!},\;\;\; k\geq 1.
\]
If also, as it is well known, $(\sigma)_n$ stands for the
Pochhammer symbol defined by
\[
(\sigma)_n=\frac{\Gamma(\sigma+n)}{\Gamma(\sigma)}= \left\{
\begin{array}{ll}
1&\mbox{if $n=0$},\\
\sigma(\sigma+1)...(\sigma+n-1)&\mbox{if $n\geq 1$.}
\end{array}
\right.
\]
then we can write the multiplier as
$(\sigma-(n-1))_n/(\sigma+k-(n-1))_n$ and throughout this paper we
represent this fraction by $[\sigma]_{n/k}$. Thus we have
\begin{equation}
c_{n,k}^\sigma=\frac{(\sigma-(n-1))_n}{(\sigma+k-(n-1))_n}c_k=[\sigma]_{n/k}c_k\,
\label{5}
\end{equation}
with $[\sigma]_{0/k}=1$. By setting
$p_{\sigma,0}(z)=L_0(z)=(1+z)/(1-z)$ we see easily that the {\em
sigma-nth} integral iteration of the Mobius functions is
\begin{equation}
L_{\sigma,n}(z)=\frac{\sigma-(n-1)}{z^{\sigma-(n-1)}}\int_0^zt^{\sigma-n}L_{\sigma,n-1}(t)dt,\;\;\;
n\geq 1\,. \label{6}
\end{equation}
The function $L_{\sigma,n}(z)$ will play a cental role in the
family $P_n^\sigma$ similar to the role of the Mobius function
$L_0(z)$ in the family $P$.
Now from ~(\ref{5}) and the fact that
$|c_k|\leq 2$ (Caratheodory lemma), we have the following
inequality
\begin{equation}
|c_{n,k}^\sigma|\leq 2[\sigma]_{n/k},\;\;\; k\geq 1\, \label{7}
\end{equation}
with equality if and only if $p_{\sigma,n}(z)=L_{\sigma,n}(z)$
given by ~(\ref{6}).
\begin{remark}
From Definitions 4 and 5 we note that $P_1^\sigma=P_1$.
\end{remark}
The following results characterizing the family $P_n^\sigma$ can
be obtained {\em mutatis mutandis} as in Section 2 of \cite{BO},
thus we omit the proofs.
\begin{theorem}
Let $\gamma\neq 1$ be a nonnegative real number. Then for any
fixed $\sigma$ and each $n\geq 1$
\[
Re\;p_{\sigma,n-1}(z)>\gamma\Rightarrow
Re\;p_{\sigma,n}(z)>\gamma,\;\;\; 0\leq\gamma<1,
\]
and
\[
Re\;p_{\sigma,n-1}(z)<\gamma\Rightarrow
Re\;p_{\sigma,n}(z)<\gamma,\;\;\; \gamma>1.
\]
\end{theorem}
\begin{corollary}
$P_n^\sigma\subset P$, $n\geq 1$.
\end{corollary}
\begin{theorem}
$P_{n+1}^\sigma\subset P_n^\sigma$, $n\geq 1$.
\end{theorem}
\begin{theorem}
Let $p_{\sigma,n}\in P_n^\sigma$. Then
\[
{\rm(a)} |p_{\sigma,n}(z)|\leq 1+2\sum_{k=1}^\infty
[\sigma]_{n/k}r^k,\;\;\;|z|=r,
\]
\[
{\rm(b)} Re\;p_{\sigma,n}(z)\geq 1+2\sum_{k=1}^\infty
[\sigma]_{n/k}(-r)^k,\;\;\;|z|=r.
\]
The results are sharp for the function
$p_{\sigma,n}(z)=L_{\sigma,n}(z)$ in the upper bound and
$p_{\sigma,n}(z)=L_{\sigma,n}(-z)$ in the lower bound.
\end{theorem}
\begin{corollary}
$p_{\sigma,n}\in P_n^\sigma$ if and only if $p_{\sigma,n}(z)\prec
L_{\sigma,n}(z)$.
\end{corollary}
\begin{remark}
If we choose $n=0$ in the corollary above we see that $p\in P$ if
and only if $p(z)\prec L_0(z)$ which is well known.
\end{remark}
\begin{remark}
For $z\in E$, the following are equivalent:

{\rm(i)} $p\prec L_0(z)$,

{\rm(ii)} $p\in P$,

{\rm(iii)} $p_{\sigma,n}\in P_n^\sigma$,

{\rm(iv)} $p_{\sigma,n}(z)\prec L_{\sigma,n}(z)$.
\end{remark}
\begin{theorem}
$P_n^\sigma$ is a convex set.
\end{theorem}
\begin{proof}
Let $p_{\sigma,n},\;q_{\sigma,n} \in P_n^\sigma$. Then for
nonnegative real numbers $\mu_1$ and $\mu_2$ with $\mu_1+\mu_2=1$,
we have
\[
\mu_1p_{\sigma,n}+\mu_2q_{\sigma,n}=\frac{\sigma-(n-1)}{z^{\sigma-(n-1)}}
\int_0^zt^{\sigma-n}(\mu_1p_{\sigma,n-1}+\mu_2q_{\sigma,n-1})(t)dt.
\]
The result follows inductively since
$\mu_1p_{\sigma,0}+\mu_2q_{\sigma,0}=\mu_1p(z)+\mu_2q(z)\in P$,
for $p,\;q\in P$.
\end{proof}

\section{Characterizations of the class $B_n^\sigma(\beta)$}
In this section we present the main results of this work. These
include inclusion, growth, covering, distortion, closure under
certain integral transformation and coefficient inequalities.

First we prove the following lemma, similar to Lemma 1.
\begin{lemma}
Let $f\in A$ and $\alpha$, $\beta$ and $D^n$ as defined above.
Then the following are equivalent:

{\rm(i)} $f\in B_n^\sigma(\beta)$,

{\rm(ii)} $(L_n^\sigma f(z)/z-\beta)/(1-\beta)\in P$,

{\rm(iii)} $(f(z)/z-\beta)/(1-\beta)\in P_n^\sigma$.
\end{lemma}
\begin{proof}
That (i) $\Leftrightarrow$ (ii) is clear from Definition 5. Now
(ii) is true $\Leftrightarrow$ there exists $p\in P$ such that
\begin{equation}
\begin{split}
L_n^\sigma f(z)=z[\beta+(1-\beta)p(z)]\\
&=z+(1-\beta)\sum_{k=1}^\infty c_kz^{k+1}\,. \label{8}
\end{split}
\end{equation}
Applying the operator $l_n^\sigma$ on ~(\ref{8}), we have
~(\ref{8}) $\Leftrightarrow$
\[
f(z)=z+(1-\beta)\sum_{k=1}^\infty c_{n,k}^\sigma z^{k+1}
\]
$\Leftrightarrow$
\begin{equation}
\frac{f(z)/z-\beta}{1-\beta}=1+\sum_{k=1}^\infty c_{n,k}^\sigma
z^k\,. \label{9}
\end{equation}
The right hand side of ~(\ref{9}) is a function in $P_n^\sigma$.
This proves the lemma.
\end{proof}
Now the main results.
\begin{theorem}
For any fixed $\sigma$ satisfying $\sigma-(n-1)>0$, the following
inclusion holds
\[
B_{n+1}^\sigma(\beta)\subset B_n^\sigma(\beta),\;\;\;n\in \N.
\]
\end{theorem}
\begin{proof}
Let $f\in B_{n+1}^\sigma(\beta)$. Then by Lemma 2,
$(f(z)/z-\beta)/(1-\beta)\in P_{n+1}^\sigma$. By Theorem 3,
$(f(z)/z-\beta)/(1-\beta)\in P_n^\sigma$. That is, by Lemma 2,
again $f\in B_n^\sigma(\beta)$.
\end{proof}
\begin{theorem}
The class $B_1^\sigma(\beta)$ consists only of univalent functions
in $E$.
\end{theorem}
\begin{proof}
Let $f\in B_1^\sigma(\beta)$. Then by Lemma 2,
$(f(z)/z-\beta)/(1-\beta)\in P_1^\sigma$. Since $\sigma$ is any
fixed integer satisfying $\sigma-(n-1)>0$, we have $\sigma>0$ for
$n=1$ and by Remark 2, it follows that
$(f(z)/z-\beta)/(1-\beta)\in P_1$. Thus by Lemma 1, this implies
that the function $f(z)$ belongs to the class
$T_1^\sigma(\beta)(\equiv T_1^\alpha(\beta))$ which consists only
of univalent functions in $E$.
\end{proof}
From Theorems 5 and 6 we have
\begin{corollary}
For $n\geq 1$, $B_n^\sigma(\beta)$ consists only of univalent
functions in $E$.
\end{corollary}
\begin{theorem}
Let $f\in B_n^\sigma(\beta)$. Then we have the sharp inequalities
\[
|a_k|\leq 2(1-\beta)[\sigma]_{n/(k-1)},\;\;\;k\geq 2.
\]
Equality is attained for
\begin{equation}
f(z)=z+2(1-\beta)\sum_{k=2}^\infty [\sigma]_{n/(k-1)}z^k\,.
\label{10}
\end{equation}
\end{theorem}
\begin{proof}
The result follows from equation ~(\ref{9}) and the inequality
~(\ref{7}).
\end{proof}
\begin{theorem}
The class $B_n^\sigma(\beta)$ is closed under the Bernard integral
\begin{equation}
F(z)=\frac{c+1}{z^c}\int_0^zt^{c-1}f(t)dt,\;\;\;c+1>0\,.
\label{11}
\end{equation}
\end{theorem}
\begin{proof}
From ~(\ref{11}) we have
\begin{equation}
\frac{F(z)/z-\beta}{1-\beta}=\frac{\nu}{z^\nu}\int_0^zt^{\nu-1}\left(\frac{f(t)/t-\beta}{1-\beta}\right)dt\,
\label{12}
\end{equation}
where $\nu=c+1$. Since $f\in B_n^\sigma(\beta)$, taking
$\nu=c+1=\sigma-n$, we can write ~(\ref{12}) as
\[
\frac{F(z)/z-\beta}{1-\beta}=\frac{\sigma-n}{z^{\sigma-n}}\int_0^zt^{(\sigma-n)-1}p_{\sigma,n}(t)dt
\]
which implies that $(F(z)/z-\beta)/(1-\beta)\in P_{n+1}^\sigma$.
Thus by Theorem 2, we have $(F(z)/z-\beta)/(1-\beta)\in
P_n^\sigma$. Hence $F\in B_n^\sigma(\beta)$.
\end{proof}
\begin{theorem}
Let $f\in B_n^\sigma(\beta)$. Then
\[
r+2(1-\beta)\sum_{k=2}^\infty (-1)^{k-1}[\sigma]_{n/(k-1)}r^k\leq
|f(z)|\leq r+2(1-\beta)\sum_{k=2}^\infty [\sigma]_{n/(k-1)}r^k.
\]
The inequalities are sharp.
\end{theorem}
\begin{proof}
The result follows by taking
$p_{\sigma,n}(z)=(f(z)/z-\beta)(1-\beta)$ in Theorem 3. Upper
bound equality is realized for the function given by ~(\ref{10})
while equality in the lower bound equality is attained for the
function
\begin{equation}
f(z)=z+2(1-\beta)\sum_{k=2}^\infty
(-1)^{k-1}[\sigma]_{n/(k-1)}z^k\,. \label{13}
\end{equation}
This completes the proof.
\end{proof}
\begin{theorem}
Each function $f(z)$ in the class $B_n^\sigma(\beta)$ maps the
unit disk onto a domain which covers the disk
$|w|<1+2(1-\beta)\sum_{k=2}^\infty (-1)^{k-1}[\sigma]_{n/(k-1)}$.
The result is sharp.
\end{theorem}
\begin{proof}
From Theorem 9, we have $|f(z)|\geq r+2(1-\beta)\sum_{k=2}^\infty
(-1)^{k-1}[\sigma]_{n/(k-1)}r^k$. This implies that the range of
every function $f(z)$ in the class $B_n^\sigma(\beta)$ covers the
disk
\begin{multline*}
|w|<1+2(1-\beta)\sum_{k=2}^\infty (-1)^{k-1}[\sigma]_{n/(k-1)}\\
=\inf_{r\to 1}\left\{r+2(1-\beta)\sum_{k=2}^\infty
(-1)^{k-1}[\sigma]_{n/(k-1)}r^k\right\}.
\end{multline*}
The function given by ~(\ref{13}) shows that the result is sharp.
\end{proof}
\begin{theorem}
Let $f\in B_n^\sigma(\beta)$. Define
\[
M(\sigma,n,\beta,r)=\sigma-(n-1)+2(1-\beta)\sum_{k=1}^\infty
[\sigma]_{(n-1)/k}r^k
\]
and
\[
m(\sigma,n,\beta,r)=\sigma-(n-1)+2(1-\beta)\sum_{k=1}^\infty
[\sigma]_{(n-1)/k}(-r)^k
\]
with
\[
[\sigma]_{(-1)/k}=\frac{\sigma+k+1}{\sigma+1}.
\]
Then
\[
m(\sigma,n,\beta,r)\leq\left|(\sigma-n)\frac{f(z)}{z}+f^{\prime}(z)\right|\leq
M(\sigma,n,\beta,r).
\]
The inequalities are sharp.
\end{theorem}
\begin{proof}
Since $f\in B_n^\sigma(\beta)$, by Lemma 2, there exists
$p_{\sigma,n}\in P_n^\sigma$ such that
\begin{equation}
f(z)=z[\beta+(1-\beta)p_{\sigma,n}(z)]\,. \label{14}
\end{equation}
Hence we have
\begin{equation}
f^{\prime}(z)=\beta+(1-\beta)[p_{\sigma,n}(z)+zp^{\prime}_{\sigma,n}(z)]\,.
\label{15}
\end{equation}
From ~(\ref{14}) and ~(\ref{15}) we get
\begin{equation}
(\sigma-n)\frac{f(z)}{z}+f^{\prime}(z)=(\sigma-(n-1))\beta+(1-\beta)[(\sigma-(n-1))p_{\sigma,n}+zp^{\prime}_{\sigma,n}]\,.
\label{16}
\end{equation}
However we find from ~(\ref{3}) that
\[
(\sigma-(n-1))p_{\sigma,n}(z)+zp^{\prime}_{\sigma,n}(z)=(\sigma-(n-1))p_{\sigma,n-1}(z)
\]
so that ~(\ref{16}) becomes
\[
(\sigma-n)\frac{f(z)}{z}+f^{\prime}(z)=(\sigma-(n-1))[\beta+(1-\beta)p_{\sigma,n-1}].
\]
Therefore by Theorem 3, we get
\begin{equation}
\left|(\sigma-n)\frac{f(z)}{z}+f^{\prime}(z)\right|\leq
\sigma-(n-1)+2(1-\beta)\sum_{k=1}^\infty [\sigma]_{(n-1)/k}r^k\,
\label{17}
\end{equation}
and
\begin{equation}
Re\left\{(\sigma-n)\frac{f(z)}{z}+f^{\prime}(z)\right\}\geq
\sigma-(n-1)+2(1-\beta)\sum_{k=1}^\infty
[\sigma]_{(n-1)/k}(-r)^k\,. \label{18}
\end{equation}
The inequalities now follow from ~(\ref{17}) and ~(\ref{18}).
Upper bound equality is realized for the function $f(z)$ given by
~(\ref{10}) while equality in the lower bound equality is attained
for the function $f(z)$ defined by ~(\ref{13}).
\end{proof}
Finally we prove
\begin{theorem}
$B_n^\sigma(\beta)$ is a convex family of analytic and univalent
functions.
\end{theorem}
\begin{proof}
Let $f,\;g\in B_n^\sigma(\beta)$. Then by Lemma 2 there exists
$p_{\sigma,n},\;q_{\sigma,n}\in P_n^\sigma$ such that
\[
f(z)=z[\beta+(1-\beta)p_{\sigma,n}(z)]
\]
and
\[
g(z)=z[\beta+(1-\beta)q_{\sigma,n}(z)].
\]
Therefore for nonnegative real numbers $\mu_1$ and $\mu_2$ with
$\mu_1+\mu_2=1$, we have
\begin{multline*}
h(z)=\mu_1f(z)+\mu_2g(z)=z\mu_1[\beta+(1-\beta)p_{\sigma,n}(z)]+z\mu_2[\beta+(1-\beta)q_{\sigma,n}(z)]\\
=z[(\mu_1+\mu_2)\beta+(1-\beta)(\mu_1p_{\sigma,n}+\mu_2q_{\sigma,n}]\\
=z[\beta+(1-\beta)(\mu_1p_{\sigma,n}+\mu_2q_{\sigma,n}].
\end{multline*}
The proof completes with Theorem 4.
\end{proof}

\section{General Remarks}
The two-parameter integral iteration of the Caratheodory functions
presented in Section 2 of this paper has also proved very
resourceful in providing elegantly short proofs of many
fundamental results in the theory of analytic and univalent
functions. An earlier one presented in \cite{BO} closely relates
with certain classes of functions defined by the Salagean
derivative operators. This provides the motivation to search for
analogous iteration that will equivalently closely relate with
certain other classes of functions involving the Ruscheweyh
derivative, and this leads us to defining new operators
$L_n^\sigma\colon A\to A$, which includes the Ruscheweyh
derivative as a special case.

Finally we remark that the results presented in this work include
many earlier ones as particular cases.
 \medskip

{\it Acknowledgements.} The author is indebted to his colleague,
Fransesco Sarnari, a Doctoral Student of the Department of
Mathematics, University of Leeds, Leeds, United Kingdom, for his
tremendous assistance during the preparation of this work. {\em
"Arriving Gaza wouldn't have been more gracious"}.

\vspace{10pt}

\hspace{-4mm}{\small{Received}}

% ADDRESS
\vspace{-12pt}
\ \hfill \
\begin{tabular}{c}
{\small\em  Department of Mathematics}\\
{\small\em  University of Ilorin}\\
{\small\em  Ilorin, Nigeria}\\
{\small\em E-mail: {\tt abuuabdilqayyuum@gmail.com}} \\
\end{tabular}

\end{document}